\documentclass[12pt]{article}
\def\q{\hfill\rule{1ex}{1ex}}
\def\0{\emptyset}

\def\q{\hfill\rule{1ex}{1ex}}

\def\n{\noindent}

\usepackage{authblk}
\usepackage{amsfonts}
\usepackage{enumitem}
\usepackage{calc}
\usepackage{amsthm}
\usepackage{amssymb}
\usepackage{amsmath}
\usepackage{tikz}
\usepackage{pgfplots}
\usepackage{graphicx}

\newtheorem{definition}{Definition}
\newtheorem{theorem}{Theorem}
\newtheorem{lemma}{Lemma}
\newtheorem{corollary}{Corollary}

\newtheorem{proposition}{Proposition}
\newtheorem{conjecture}{Conjecture}
\newtheorem{problem}{Problem}

\begin{document}
\title{\bf The anti-Ramsey number of $C_{3}$ and $C_{4}$ in the complete $r$-partite graphs}

\author[1,2,4]{{\small\bf Chunqiu Fang}\thanks{ fcq15@tsinghua.org.cn, supported in part by CSC(No. 201806210164) and NSFC (No. 11771247).}}
\author[2]{{\small\bf Ervin Gy\H ori}\thanks{ gyori.ervin@renyi.mta.hu, supported in part by the National Research, Development and Innovation Office NKFIH, grants K116769, K117879, K126853 and K132696.}}
\author[3,5]{{\small\bf Binlong Li}\thanks{ libinlong@mail.nwpu.edu.cn.}}
\author[2,3,5]{{\small\bf Jimeng Xiao}\thanks{ xiaojimeng@mail.nwpu.edu.cn, supported in part by CSC(No. 201706290171).}}

\affil[1]{Department of Mathematical Sciences, Tsinghua University, Beijing 100084, China }
\affil[2]{Alfr\'ed R\'enyi Institute of Mathematics, Hungarian Academy of Sciences, Re\'altanoda u.13-15, 1053 Budapest, Hungary}
\affil[3]{Department of Applied Mathematics, Northwestern Polytechnical University, Xi'an, China}
\affil[4]{Yau Mathematical Sciences Center, Tsinghua University, Beijing 10084, China}
\affil[5]{Xi'an-Budapest Joint Research Center for Combinatorics, Northwestern Polytechnical University, Xi'an, China}

\date{}
\maketitle

\begin{abstract}
A subgraph of an edge-colored graph is rainbow, if all of its edges have different colors. For a graph $G$ and a family $\mathcal{H}$ of graphs, the anti-Ramsey number $ar(G, \mathcal{H})$ is the maximum number $k$ such that there exists an edge-coloring of $G$ with exactly $k$ colors  without rainbow copy of any graph in $\mathcal{H}$. In this paper, we study the anti-Ramsey number of $C_{3}$ and $C_{4}$ in the complete $r$-partite graphs. For $r\ge 3$ and $n_{1}\ge n_{2}\ge \cdots\ge n_{r}\ge 1$, we determine $ ar(K_{n_{1}, n_{2}, \ldots, n_{r}},\{C_{3}, C_{4}\}), ar(K_{n_{1}, n_{2}, \ldots, n_{r}}, C_{3})$
 and $ar(K_{n_{1}, n_{2}, \ldots, n_{r}}, C_{4})$.

\end{abstract}

~

\n{\large\bf 1. Introduction}
~\\

We call a subgraph of an edge-colored graph {\em rainbow}, if all the edges have different colors. For a graph $G$ and a family $\mathcal{H}$ of graphs, the anti-Ramsey number $ar(G, \mathcal{H})$ is the maximum number $k$ such that there exists an edge-coloring of $G$ with exactly $k$ colors  without rainbow copy of any graph in $\mathcal{H}$. If $\mathcal{H}=\{H\}$, then we denote $ar(G,\{H\})$ by $ar(G,H)$.  The study of anti-Ramsey theory was initiated by Erd\H os, Simonovits and S\'os \cite{Erdos} and considered in the classical case when $G=K_{n}$. Since then plentiful results were established for a variety of graphs $H$, including among cycles \cite{Alon, Montellano1}, cliques \cite{Erdos, Montellano2, Schiermeyer}, trees \cite{Jiang1, Jiang2}, and non-connected graphs \cite{Gilboa, Yuan}. Also, different graphs were considered as the underlying host graph $G$, including complete bipartite graphs \cite{Axenovich2}, complete hypergraphs \cite{Gu, Ozkahya},  complete split graphs \cite{Gorgol, Jin}, triangulations \cite{Chen1, Lan}. See Fujita, Magnant, and Ozeki \cite{Fujita} for a survey.

For cycles, Erd\H os, Simonovits and S\'os \cite{Erdos} showed that $ar(K_{n}, C_{3})=n-1$ and conjectured that $ar(K_{n}, C_{k})=(\frac{k-2}{2}+\frac{1}{k-1})n+O(1)$,  for all $n\ge k\ge 3$. Alon \cite{Alon} proved it for $k=4$ by showing that $ar(K_{n}, C_{4})=\left\lfloor\frac{4n}{3}\right\rfloor-1.$ Jiang, Schiermeyer and West \cite{Jiang3} proved this conjecture for $k\le 7$. Finally, Montellano-Ballesteros and Neumann-Lara \cite{Montellano1} completely proved this conjecture.

\begin{theorem}(\cite{Montellano1})
For all $n\ge k\ge 3$, let $n\equiv r_{k} \pmod{(k-1)}$, $0\le r_{k}\le k-2$, we have
$$ar(K_{n}, C_{k})=\left\lfloor\frac{n}{k-1}\right\rfloor\binom{k-1}{2}+\binom{r_{k}}{2}+\left\lceil\frac{n}{k-1}\right\rceil-1.$$
\end{theorem}

Axenovich, Jiang and K\" undgen \cite{ Axenovich2} considered the even cycles in complete bipartite graphs and proved the following result.

\begin{theorem}(\cite{Axenovich2})
For $n\ge m\ge 1$ and $k\ge 2$,
\begin{align} \nonumber
ar(K_{m,n}, C_{2k})=
 \begin{cases}
(k-1)(m+n)-2(k-1)^{2}+1,\,& m\ge 2k-1;\\
(k-1)n+m-(k-1),\,& k-1\le m\le 2k-1;\\
mn, & m\le k-1.
 \end{cases}
 \end{align}

\end{theorem}

A {\em complete split graph} $K_{n}+\overline{K_{s}}$ is a {\em join} of a complete graph $K_{n}$ and an empty graph $\overline{K_{s}}$,  that is the graph obtained from $K_{n}\cup \overline{K_{s}}$ by joining each vertex of $K_{n}$ with each vertex of $\overline{K_{s}}$.  Gorgol \cite{Gorgol} considered the cycles in complete split graphs and proved the following result.

\begin{theorem}(\cite{Gorgol})
For $n\ge 2$, $s\ge 1$,  $ar(K_{n}+\overline{K_{s}}, C_{3})=n+s-1.$
\end{theorem}

They also gave a lower bound and an upper bound for $ar(K_{n}+\overline{K_{s}}, C_{4})$ and conjectured that the exact value is closer to the lower bound.

\begin{theorem}(\cite{ Gorgol})
For $s\ge n\ge 4$, $\left\lfloor\frac{4n}{3}\right\rfloor+s-1\le ar(K_{n}+\overline{K_{s}}, C_{4})\le \left\lfloor\frac{7n}{3}\right\rfloor+s-3.$

\end{theorem}

In section 2, we study the anti-Ramsey number of $\{C_{3}, C_{4}\}$ and $C_{3}$ in the complete $r$-partite graphs and prove the following two theorems.

\begin{theorem}
For $r\ge 3$ and $n_{1}\ge n_{2}\ge \cdots\ge n_{r}\ge 1$, we have
$$ ar(K_{n_{1}, n_{2}, \ldots, n_{r}}, \{C_{3},C_{4}\})=n_{1}+n_{2}+\cdots+n_{r}-1.$$
\end{theorem}

\begin{theorem}
For $r\ge 3$ and $n_{1}\ge n_{2}\ge \cdots\ge n_{r}\ge 1$, we have
 \begin{align}\nonumber
  ar(K_{n_{1}, n_{2}, \ldots, n_{r}}, C_{3})=
 \begin{cases}
 n_{1}n_{2}+n_{3}n_{4}+ \cdots+n_{r-2}n_{r-1}+n_{r}+\frac{r-1}{2}-1,& r\, \text{ is odd};\\
n_{1}n_{2}+n_{3}n_{4}+ \cdots+n_{r-1}n_{r}+\frac{r}{2}-1,& r \,\text{ is even}.
 \end{cases}
 \end{align}
\end{theorem}

In section 3, we generalize the theorem of Alon \cite{Alon} to complete $r$-partite graphs. We call two subgraphs $H_{1}$ and $H_{2}$ of $G$ are {\em independent} if they are vertex disjoint.

\begin{theorem}

For $r\ge 3$ and $n_{1}\ge n_{2}\ge\cdots\ge n_{r}\ge 1$, we have
$$ar(K_{n_{1}, n_{2}, \ldots, n_{r}}, C_{4})= n_{1}+n_{2}+\cdots+n_{r}+t-1,$$
where $t=\min\left\{\left\lfloor\frac{\sum_{i=1}^{r}n_{i}}{3}\right\rfloor, \left\lfloor\frac{\sum_{i=2}^{r}n_{i}}{2}\right\rfloor, \sum_{i=3}^{r}n_{i}\right\}$ is the maximum number of independent triangles of $K_{n_{1}, n_{2}, \ldots, n_{r}}$.

\end{theorem}

Notice that $K_{n}+\overline{K_{s}}$ is a complete $(n+1)$-partite graph $K_{s, 1, \ldots, 1}$. If $2s \ge n \ge 4$, then the number of independent triangles of $K_{n}+\overline{K_{s}}$ is $\left\lfloor\frac{n}{2}\right\rfloor$. We have the following corollary and this answers a question of Gorgol \cite{Gorgol}.

\begin{corollary}

For $2s\ge n\ge 4$, we have $ar(K_{n}+\overline{K_{s}}, C_{4})=\left\lfloor\frac{3n}{2}\right\rfloor+s-1.$

\end{corollary}

The paper is organized as follows. In Section 2, we will give the proof of Theorem 5 and Theorem 6. In Section 3, we will prove Theorem 7. Finally we will give some open problems in Section 4.

{\bf Notations:}  Let $G $ be a simple undirected graph. For $x\in V (G)$, we denote the {\em neighborhood} and the {\em degree} of $x$ in $G$ by $N_G(x)$ and $d_G(x)$, respectively.  The {\em maximum degree}  of $G$ is denoted by  $\Delta(G)$. We will use $G- x$ to denote the graph that arises from $G$ by deleting the vertex $x\in V (G)$. For $\emptyset\not= X \subset V (G)$, $G[X]$ is the subgraph of $G$ induced by $X$ and $G -X$ is the subgraph of $G$ induced by $V (G)\setminus X$. Given a graph $G= (V, E)$, for any (not necessarily disjoint) vertex sets $A, B\subset V$, we let $E_{G}(A, B)=\{uv\in E(G): u\ne v, u\in A, v\in B\}$. A vertex $v$ of a graph $G$ is called a {\em cut vertex} if the components number of $G-x$ is bigger than $G$. A {\em block} of a graph is a subgraph which is connected, has no cut vertex and is maximal with respect to this property.


Given an edge-coloring $c$ of $G$, we denote the color of an edge $uv$ by $c(uv)$.  A color $a$ is  {\em starred} (at $x$) if all the edges with color $a$ induce a star $K_{1,r}$ (centered at the vertex $x$). Note that the two vertices of $K_{1,1}$ can be regarded as the center of  $K_{1,1}$, if a color $a$ is stared at both $x$ and $y$, then $xy$ is the unique edge with color $a$. We let $d^{c}(v)=|\{a\in C(v) : a$ is starred at $ v\}|$. For a subgraph $H$ of $G$, we denote $C(H)=\{c(uv):\, uv\in E(H)\}$.
A {\em representing subgraph} in an edge-coloring of $K_n$ is a spanning subgraph
containing exactly one edge of each color.

~\\

\n{\large\bf {2. 3-cycle}}

~\\

In this section, we will prove Theorem 5 and Theorem 6. We recall the proof  of $ar(K_{n}, C_{3})=n-1$ in \cite{Erdos} because we will use the idea of it later. Let $V(K_{n})=\{v_{1}, v_{2}, \ldots, v_{n}\}.$ For $1\le i<j\le n$, we color edge $v_{i}v_{j}$ with color $i$. In such a way, we use exactly $n-1$ colors and there is no rainbow $C_{3}$. On the other hand, for any $n$-edge-coloring of $K_{n}$, we take a representing subgraph $G$, then $G$ contains a rainbow cycle since $|E(G)|=n$. By adding one chord to the cycle, we get two shorter cycles and at least one of them is rainbow. Do the same operation on the shorter rainbow cycle.  Finally, we will find a rainbow $C_{3}$. To generalize the idea of the proof, we first give the following useful definition and lemma.

\begin{definition}
Let $r\ge 3$ and $G$ be an $r$-partite graph with parts $V_{1}, V_{2},\ldots, V_{r}$, we call a subgraph $H$ of $G$ {\em multipartite}, if there are at least three distinct parts $V_{i}, V_{j}, V_{k}$ such that  $V(H)\cap V_{i}\ne \0, V(H)\cap V_{j}\ne \0$ and $V(H)\cap V_{k}\ne \0$.
\end{definition}

\begin{lemma}
Let $r\ge 3$ and $n_{1}\ge n_{2}\ge \cdots\ge n_{r}\ge 1$.
For an edge-coloring of $K_{n_{1}, n_{2},\ldots, n_{r}}$, if there is a rainbow multipartite cycle, then there is a rainbow $C_{3}$.
\end{lemma}

\begin{proof}
Since for any multipartite cycle $C$ of length at least $4$, there is a chord $e$ of $C$ such that $C\cup e$ contains two multipartite cycles $C^{1}$ and $C^{2}$, where $E(C^{1})\cap E(C^{2})=\{e\}$ and  $E(C^{1})\cup E(C^{2})=E(C)\cup\{e\}$. Thus for any rainbow multipartite cycle of length at least $4$, we can find a shorter rainbow multipartite cycle. Do the same operation on the shorter rainbow multipartite cycle.  Finally, we will find a rainbow $C_{3}$.
\end{proof}

We determine the anti-Ramsey number of $\{C_{3}, C_{4}\}$ in the complete $r$-partite graphs as follows.

\noindent{\bf Theorem 5.}  For $r\ge 3$ and $n_{1}\ge n_{2}\ge \cdots\ge n_{r}\ge 1$, we have
$$ ar(K_{n_{1}, n_{2}, \ldots, n_{r}}, \{C_{3},C_{4}\})=n_{1}+n_{2}+\cdots+n_{r}-1.$$

\begin{proof}
{\bf Lower bound:} Let $n=n_{1}+n_{2}+\cdots+n_{r}$. For all $1\le i\le r$, we take a vertex $u_{i}$ from $V_{i}$.  Let $U=\{u_{1}, u_{2}, \ldots, u_{r}\}$ and $V\setminus U=\{u_{r+1}, u_{r+2}, \ldots, u_{n}\}$. Note that for all $1\le i\le n-1$, there is at least one edge between $u_{i+1}$ and $\{u_{j}: 1\le j\le i\}$ in $K_{n_{1}, n_{2}, \ldots, n_{r}}$. For $1\le i\le n-1$, we color all the edges between $u_{i+1}$ and $\{u_{j}: 1\le j\le i\}$ by color $i$. In such a way, we use exactly $n-1$ colors, and there is no rainbow $C_{3}$ or $C_{4}$.

{\bf Upper bound:} For an $(n_{1}+n_{2}+\cdots+n_{r})$-edge-coloring of $K_{n_{1}, n_{2}, \ldots, n_{r}}$, we take a representing subgraph $G$. Since $|E(G)|=n_{1}+n_{2}+\cdots+n_{r}=|V(G)|$, we can find a cycle $C$ of $G$ and $C$ is rainbow. If $C$ is multipartite, there is a rainbow $C_{3}$ by Lemma  1.  If $C$ is not multipartite, then there exist $i$ and $j$ where $1\le i < j \le r$, such that $V(C)\subset V_{i}\cup V_{j}$ and $|V(C)\cap V_{i}|=|V(C)\cap V_{j}|$. If the length of $C$ is at least $6$, for a chord $e$ of $C$, $C\cup e$ contains two even cycles $C^{1}$ and $C^{2}$, where $E(C^{1})\cap E(C^{2})=\{e\}$ and $E(C^{1})\cup E(C^{2})=E(C)\cup\{e\}$. Since $C$ is rainbow, we claim that at least one of $C^{1}$ and $C^{2}$ is a rainbow even cycle. Thus, we find a rainbow shorter even cycle. Do the same operation on the shorter rainbow even cycle. Finally, we can find a rainbow $C_{4}$.

\end{proof}

Now we will determine the maximum number of edges in an $r$-partite graphs without multipartite cycles by the following lemmas.

\begin{lemma}
For $r\ge 3$ and $n_{1}\ge n_{2}\ge \cdots\ge n_{r}\ge 1$, if $G\subset K_{n_{1}, n_{2}, \ldots, n_{r}}$ and $G$ contains no multipartite $P_{3}$, then
\begin{align}\nonumber
  |E(G)|\le
 \begin{cases}
n_{1}n_{2}+n_{3}n_{4}+\cdots+n_{r-2}n_{r-1}, & r\,\text{ is odd};\\
n_{1}n_{2}+n_{3}n_{4}+\cdots+n_{r-1}n_{r}, & r \,\text{ is even}.
 \end{cases}
 \end{align}
\end{lemma}

\begin{proof} We will prove it by induction on $r$. It is obvious to see that the conclusion holds for the base cases $r=1$ and $r=2$. Assume that it holds for all integers less than $r$, and let $G\subset K_{n_{1}, n_{2},\ldots,n_{r}}$ which contains no multipartite $P_{3}$.  For any vertex $v$, there is some part $V_{i}$ such that $N_{G}(v)\subset V_{i}$. Choose $v_{1}\in V_{1}$ such that $d_{G}(v_{1})=\max\{d_{G}(v): v\in V_{1}\}$. Suppose $N_{G}(v_{1})\subset V_{i_{0}}$, where $2\le i_{0}\le r$.  For every vertex $v\in V_{1} \setminus \{v_{1}\}$, we delete all the edges incident with $v$ in $G$ and connect all the edges between $v$ and $N_{G}(v_{1})$. For every vertex $u\in V_{i_{0}}\setminus N_{G}(v_{1})$, we delete all the edges incident with $u$ in $G$ and connect all the edges between $u$ and $V_{1}$. Denote the graph we obtained by $G_{1}$, we have $|E(G)|\le |E(G_{1})|$ and $G_{1}$ contains no multipartite $P_{3}$.

 If $i_{0}\ne 2$, we choose $v_{2}\in V_{2}$ such that $d_{G_{1}}(v_{2})=\max\{d_{G_{1}}(v):v\in V_{2}\}$. Suppose $N_{G_{1}}(v_{2})\subset V_{j_{0}}$, where $3\le j_{0}\le r$ and $j_{0}\ne i_{0}$.  For every vertex $v\in V_{2}\setminus \{v_{2}\}$, we delete all the edges incident with $v$ in $G_{1}$ and connect all the edges between $v$ and $N_{G_{1}}(v_{2})$. For every vertex $u\in V_{j_{0}}\setminus N_{G_{1}}(v_{2})$, we delete all the edges incident with $u$ in $G_{1}$ and connect all the edges between $u$ and $V_{2}$.
Denote the graph we obtained by $G_{2}$, we have $|E(G_{1})|\le |E(G_{2})|$ and $G_{2}$ contains no multipartite $P_{3}$. We note that both $G_{2}[V_{1}\cup V_{i_{0}}]$ and $G_{2}[V_{2}\cup V_{j_{0}}]$ are complete bipartite graphs. We delete all the edges of $G_{2}[V_{1}\cup V_{i_{0}}]$ and $G_{2}[V_{2}\cup V_{j_{0}}]$, connect all the edges between $V_{1}$ and $V_{2}$ and connect all the edges between $V_{i_{0}}$ and $V_{j_{0}}$. Denote the graph we obtained by $G'$, then $G'$  contains no multipartite $P_{3}$. Since $n_{1}n_{i_{0}}+n_{2}n_{j_{0}}\le n_{1}n_{2}+n_{i_{0}}n_{j_{0}}$, we have $|E(G_{2})|\le |E(G')|$. If $i_{0}= 2$, we denote $G'=G_{1}$.

By the induction hypothesis, we have
 \[\begin{split} \nonumber
 |E(G)|\le |E(G')|& =n_{1}n_{2}+ |E(G'-(V_{1}\cup V_{2}))|\\
 &\le n_{1}n_{2}+
 \begin{cases}
n_{3}n_{4}+\cdots+n_{r-2}n_{r-1}, & r\,\text{ is odd};\\
n_{3}n_{4}+\cdots+n_{r-1}n_{r}, & r \,\text{ is even}.
 \end{cases}\\
 &= \begin{cases}
n_{1}n_{2}+n_{3}n_{4}+\cdots+n_{r-2}n_{r-1}, & r\,\text{ is odd};\\
n_{1}n_{2}+n_{3}n_{4}+\cdots+n_{r-1}n_{r}, & r \,\text{ is even}.
 \end{cases}
\end{split}\]

\end{proof}

\begin{lemma}

Let $G\subset K_{n_{1}, n_{2}, \ldots, n_{r}}$ which contains no multipartite cycle. If $P=xyz$ is a multipartite $P_{3}$ of $G$, then $y$ is a cut vertex of $G$.

\end{lemma}

\begin{proof}
 Suppose $y$ is not a cut vertex of $G$. Then there is a path $P'$ connect $x$ and $z$ in $G-y$, thus $xPzP'x$ is a multipartite cycle of $G$, a contradiction.
\end{proof}

\begin{lemma}

For $r\ge 3$ and $n_{1}\ge n_{2}\ge \cdots\ge n_{r}\ge 1$, if $G\subset K_{n_{1}, n_{2}, \ldots, n_{r}}$ and $G$ contains no multipartite cycle, then
\begin{align}\nonumber
|E(G)|\le
 \begin{cases}
 n_{1}n_{2}+n_{3}n_{4}+ \cdots+n_{r-2}n_{r-1}+n_{r}+\frac{r-1}{2}-1,& r\, \text{ is odd};\\
n_{1}n_{2}+n_{3}n_{4}+ \cdots+n_{r-1}n_{r}+\frac{r}{2}-1,& r \,\text{ is even}.
 \end{cases}
 \end{align}
\end{lemma}

\begin{proof}

Take a graph $G\subset K_{n_{1}, n_{2},\ldots,n_{r}}$ such that $G$ contains no multipartite cycle and $|E(G)|$ is the maximum possible.

{\bf Claim 1.} $G$ is connected and each block of $G$ is a complete bipartite graph whose two parts belong to two parts of $K_{n_{1}, n_{2}, \ldots, n_{r}}$, respectively.

{\bf Proof of  Claim 1.} If $G$ is not connected, we can add edges between the connected components of $G$, such that there is still no multipartite cycle, a contradiction. Let $B$ be a block of $G$, if there are three distinct parts $V_{i}, V_{j}$ and $V_{k}$ such that $V(B)\cap V_{i}\ne \0, V(B)\cap V_{j}\ne \0$ and $V(B)\cap V_{k}\ne \0$, then $B$ contains a multipartite cycle, a contradiction. \q

For $1\le i\le r$, we select one vertex $x_{i}\in V_{i}$ and let $V'_{i}=V_{i}\setminus \{x_{i}\}$. Denote $V_{0}=\{x_{1}, x_{2},\ldots, x_{r}\}.$ For each $i=1,2,\ldots, r$, perform  the following operations sequentially:

If there is a vertex $v\in V'_{i}$ such that $v$ is a cut vertex of $G$, and $v$ separate the blocks $B_{1}$ and $B_{2}$, then there is at least one of $B_{1}$ and $B_{2}$, we say $B_{1}$ such that there is no paths connect $x_{i}$ to $N_{G}(v)\cap V(B_{1})$ in $G-v$. Thus, we delete the edges between $v$ and $N_{G}(v)\cap V(B_{1})$, and connect the edges between $x_{i}$ and $N_{G}(v)\cap V(B_{1})$. We still denote the new graph by $G$.

{\bf Claim 2.} In each step, $G$ contains no multipartite cycle and  $|E(G)|$ is the same as in the starting graph. When the above procedure stops,  $v$ is not a cut vertex of $G$  for any vertex $v\in \cup_{i=1}^{r}V'_{i}$.

By  Claim 2 and  Lemma 3, there is no multipartite $P_{3}$ in $G[\cup_{i=1}^{r}V'_{i}]$. Therefore, by lemma 2, we have
 \begin{align}\nonumber
  |E(G[\cup_{i=1}^{r}V'_{i}])|\le
 \begin{cases}
 (n_{1}-1)(n_{2}-1)+\cdots+(n_{r-2}-1)(n_{r-1}-1),& r\,\text{ is odd};\\
 (n_{1}-1)(n_{2}-1)+\cdots+(n_{r-1}-1)(n_{r}-1),& r \,\text{ is even}.
 \end{cases}
 \end{align}
Since for all $x\in \cup_{i=1}^{r}V'_{i}$, $x$ is not a cut vertex of $G$, we have $|E_{G}(x, V_{0})|\le 1.$  Since $G[V_{0}]$ contains no cycles, we have $|E(G[V_{0}])|\le r-1$.
Thus, we have
\[\begin{split}|E(G)|&=|E(G[\cup_{i=1}^{r}V'_{i}])|+|E_{G}(\cup_{i=1}^{r}V'_{i},V_{0})|+|E(G[V_{0}])|\\
&\le \begin{cases} \begin{split} &(n_{1}-1)(n_{2}-1)+(n_{3}-1)(n_{4}-1)\cdots+(n_{r-2}-1)(n_{r-1}-1)+\\
&\,  (n_{1}-1)+(n_{2}-1)+\cdots+(n_{r}-1)+r-1,\end{split} & r\,\text{ is odd};\\
\begin{split} &(n_{1}-1)(n_{2}-1)+(n_{3}-1)(n_{4}-1)\cdots+(n_{r-1}-1)(n_{r}-1)+\\
&\,   (n_{1}-1)+(n_{2}-1)+\cdots+(n_{r}-1)+r-1,\end{split}& r \,\text{ is even}.
 \end{cases}\\
 &=  \begin{cases}
n_{1}n_{2}+n_{3}n_{4}+\cdots+n_{r-2}n_{r-1}+n_{r}+\frac{r-1}{2}-1 ,& r\,\text{ is odd};\\
n_{1}n_{2}+n_{3}n_{4}+\cdots+n_{r-1}n_{r}+\frac{r}{2}-1,& r \,\text{ is even}.
 \end{cases}
 \end{split}\]

\end{proof}

Now, we will prove Theorem 6.

\noindent{\bf Theorem 6.} For $r\ge 3$ and $n_{1}\ge n_{2}\ge \cdots\ge n_{r}\ge 1$, we have
 \begin{align}\nonumber
  ar(K_{n_{1}, n_{2}, \ldots, n_{r}}, C_{3})=
 \begin{cases}
 n_{1}n_{2}+n_{3}n_{4}+ \cdots+n_{r-2}n_{r-1}+n_{r}+\frac{r-1}{2}-1,& r\, \text{ is odd};\\
n_{1}n_{2}+n_{3}n_{4}+ \cdots+n_{r-1}n_{r}+\frac{r}{2}-1,& r \,\text{ is even}.
 \end{cases}
 \end{align}

 \begin{proof}
{\bf Lower bound:} Let $K=K_{n_{1}, n_{2}, \ldots, n_{r}}$ and $V_{1}, V_{2}, \ldots, V_{r}$ be the  partition of $V(K)$. For $r$ is odd,  first, we color $K[V_{1}\cup V_{2}], K[V_{3}\cup V_{4}], \ldots, K[V_{r-2}\cup V_{r-1}]$ rainbow. Second, for every vertex  $v\in V_{r}$, we color the edges incident $v$ with one new distinct color. Finally, for $1\le i\le \frac{r-1}{2}-1$, we color the edges between $V_{2i-1}\cup V_{2i}$ and $\cup _{j=2i+1}^{r-1}V_{j}$ with one new distinct color. In such way, we use exactly $n_{1}n_{2}+n_{3}n_{4}+ \cdots+n_{r-2}n_{r-1}+n_{r}+\frac{r-1}{2}-1$ colors, and there is no rainbow $C_{3}$.

For $r$ is even, first, we color $K[V_{1}\cup V_{2}], K[V_{3}\cup V_{4}], \ldots, K[V_{r-1}\cup V_{r}]$ rainbow. Second, for $1\le i\le \frac{r}{2}-1$, we color the edges between $V_{2i-1}\cup V_{2i}$ and $\cup _{j=2i+1}^{r}V_{j}$ with one new distinct color. In such way, we use exactly $n_{1}n_{2}+n_{3}n_{4}+ \cdots+n_{r-1}n_{r}+\frac{r}{2}-1$ colors, and there is no rainbow $C_{3}$.

{\bf Upper bound:} For $r\ge 3$ and $n_{1}\ge n_{2}\ge \cdots\ge n_{r}\ge 1$, we denote
 \begin{align}\nonumber
  f(n_{1}, n_{2},\ldots, n_{r})=
 \begin{cases}
 n_{1}n_{2}+n_{3}n_{4}+ \cdots+n_{r-2}n_{r-1}+n_{r}+\frac{r-1}{2}-1,& r\, \text{ is odd};\\
n_{1}n_{2}+n_{3}n_{4}+ \cdots+n_{r-1}n_{r}+\frac{r}{2}-1,& r \,\text{ is even}.
 \end{cases}
 \end{align}
Given any edge-coloring of $K_{n_{1}, n_{2},\ldots,n_{r}}$ with $f(n_{1}, n_{2},\ldots, n_{r})+1$ colors, we take a  representing subgraph $G$. Notice that $G$ is rainbow and $|E(G)|=f(n_{1}, n_{2},\ldots, n_{r})+1$. By Lemma 4, $G$ contains a rainbow multipartite cycle. Hence, there is a rainbow $C_{3}$ by Lemma 1.
\end{proof}

~\\

\n{\large\bf{ 3. 4-cycle}}

~\\

In this section, we study the anti-Ramsey number of $C_{4}$ in the complete $r$-partite graphs. Before doing so, we will determine the maximum number of independent triangles in the complete $r$-partite graphs by the following proposition.

\begin{proposition}

For $r\ge 3$ and $n_{1}\ge n_{2}\ge \cdots\ge n_{r}\ge 1$, let $t$ be  the maximum number of independent triangles of $K_{n_{1}, n_{2}, \ldots, n_{r}}$, then we have
$$t=\min\left\{\left\lfloor\frac{\sum_{i=1}^{r}n_{i}}{3}\right\rfloor, \left\lfloor\frac{\sum_{i=2}^{r}n_{i}}{2}\right\rfloor, \sum_{i=3}^{r}n_{i}\right\}.$$

\end{proposition}

\begin{proof}
Let $T_{1}, T_{2}, \ldots, T_{t}$ be the $t$ independent triangles of $K_{n_{1}, n_{2}, \ldots, n_{r}}$. Since each triangle contains exactly $3$ vertices of $V$,  at least two vertices of $V_{2}\cup V_{3}\cup \cdots \cup V_{r}$ and  at least one vertices of $V_{3}\cup V_{4}\cup \cdots \cup V_{r}$, we have $$t\le\min\left\{\left\lfloor\frac{\sum_{i=1}^{r}n_{i}}{3}\right\rfloor, \left\lfloor\frac{\sum_{i=2}^{r}n_{i}}{2}\right\rfloor, \sum_{i=3}^{r}n_{i}\right\}.$$

On the other hand, let $U=V\setminus (V(T_{1})\cup V(T_{2})\cup \cdots \cup V(T_{t}))$. Suppose that $t\ne\left\lfloor\frac{\sum_{i=1}^{r}n_{i}}{3}\right\rfloor$, then we have $|U|=\sum_{i=1}^{r}n_{i}-3t\ge 3.$ Since $t$ is  the maximum number of independent triangles of $K_{n_{1}, n_{2}, \ldots, n_{r}}$, there are two parts $V_{k}$ and $V_{l}$ such that $U\subset V_{k}\cup V_{l}$. Assume that $|U\cap V_{k}|\ge |U\cap V_{l}|\ge 0.$

If $|U\cap V_{l}|=0$, then $|U\cap V_{k}|= |U|-|U\cap V_{l}|\ge 3$. Since each triangle contains at least two vertices of $V\setminus V_{k}$, we have $t\ge \left\lfloor\frac{|V|-|V_{k}|}{2}\right\rfloor \ge \left\lfloor\frac{|V|-|V_{1}|}{2}\right\rfloor =\left\lfloor\frac{\sum^{r}_{i=2}n_{i}}{2}\right\rfloor$.

If $|U\cap V_{l}|= 1$, then $|U\cap V_{k}|= |U|-|U\cap V_{l}|\ge 2$. We claim that each triangle contains one vertex of $V_{k}$ and two vertices of $V\setminus V_{k}$. Otherwise, there is a triangle $T_{i}$ such that $T_{i}\cap V_{k}=\0$, and we can get two independent triangles from $V(T_{i})\cup U$, a contradiction. Thus,
$t=\frac{|V|-|V_{k}|-1}{2}$ and $\frac{|V|-|V_{k}|-1}{2}$ is an integer. We have
 $$t=\frac{|V|-|V_{k}|-1}{2}=\left\lfloor\frac{|V|-|V_{k}|}{2}\right\rfloor\ge \left\lfloor\frac{|V|-|V_{1}|}{2}\right\rfloor=\left\lfloor\frac{\sum_{i=2}^{r}n_{i}}{2}\right\rfloor.$$

If $|U\cap V_{l}|\ge 2$, we claim that each triangle contains one vertex of $V_{k}$, one vertex of $V_{l}$ and one vertex of $V\setminus(V_{k}\cup V_{l})$. Otherwise, there is a triangle $T_{i}$ such that  $|V(T_{i})\cap (V\setminus(V_{k}\cup V_{l}))|\ge 2$,  and we can get two independent triangles from $V(T_{i})\cup U$, a contradiction. Thus, $t=|V|-|V_{k}\cup V_{l}|\ge |V|-|V_{1}\cup V_{2}|=\sum_{i=3}^{r}n_{i}.$

\end{proof}


Now, we will prove Theorem 7.

\noindent{\bf Theorem 7.} For $r\ge 3$ and $n_{1}\ge n_{2}\ge\cdots\ge n_{r}\ge 1$, we have
$$ar(K_{n_{1}, n_{2}, \ldots, n_{r}}, C_{4})= n_{1}+n_{2}+\cdots+n_{r}+t-1,$$
where $t=\min\left\{\left\lfloor\frac{\sum_{i=1}^{r}n_{i}}{3}\right\rfloor, \left\lfloor\frac{\sum_{i=2}^{r}n_{i}}{2}\right\rfloor, \sum_{i=3}^{r}n_{i}\right\}$ is the maximum number of independent triangles of $K_{n_{1}, n_{2}, \ldots, n_{r}}$.

\begin{proof}

{\bf Lower bound:} Take $t$ independent triangles $T_{1}, T_{2}, \ldots, T_{t}$ and let the remaining vertices be $v_{1}, v_{2}, \ldots, v_{l}$ and $U=\cup_{j=1}^{t}V(T_{j})$, where $l=n_{1}+n_{2}+\cdots+n_{r}-3t$. First, color those triangles rainbow. Second, for $2\le i\le t$, we color all the edges between $\cup^{i-1}_{j=1}V(T_{j})$ and $V(T_{i})$ with one new color.  Finally, for $1\le i\le l$, we color all the edges between $U\cup\{v_{1}, \ldots, v_{i-1} \}$ and $v_{i}$ with one new color. In such way, we use exactly $4t-1+l=n_{1}+n_{2}+\cdots+n_{r}+t-1$ colors, and there is no rainbow copy of $C_{4}$.

{\bf Upper bound:} We will prove the upper bound by induction on $n=n_{1}+n_{2}+\cdots+n_{r}$. The base case $n_{1}= n_{2}=\cdots =n_{r}=1$ is true by Theorem 1. Assume that it holds for all integers less than $n$. For an $(n+t)$-edge coloring $c$ of $K_{n_{1}, n_{2}, \ldots, n_{r}}$, suppose there is no rainbow $C_{4}$, we have $d^{c}(x)\ge 2$ for any vertex $x\in V(K_{n_{1}, n_{2}, \ldots, n_{r}})$. Otherwise, by the induction hypothesis, we have $$|C(K_{n_{1}, n_{2}, \ldots, n_{r}}-x)|\ge n+t-1\ge ar(K_{n_{1}, n_{2}, \ldots, n_{r}}-x, C_{4})+1,$$ and there is a rainbow $C_{4}$ in $K_{n_{1}, n_{2}, \ldots, n_{r}}-x$, a contradiction.

{\bf Claim 1.} For any vertex $x\in V$, if $c(xy)$ and $c(xz)$ are two distinct colors which are starred at $x$, then $y$ and $z$ are in different parts.

{\bf Proof of Claim 1.} Assume that $y$ and $z$ are in the same part. Since $d^{c}(y)\ge 2$, there is at least one edge $yw$ such that $c(yw)$ is starred at $y$ and $c(yw)\ne c(xy)$. Note that $c(zw)\notin \{c(xy), c(xz), c(yw)\}$, thus $xywzx$ is a rainbow $C_{4}$, a contradiction. \q

{\bf Claim 2.} For any vertex $x\in V$, if $c(xy)$ is starrted at $x$, then $xy$ is the unique edge with color $c(xy)$.

{\bf Proof of Claim 2.} Assume that $xy$ is not the unique edge with color $c(xy)$, then $c(xy)$ is not starred at $y$. Since $d^{c}(x)\ge 2$, there is a vertex $z$ such that $c(xz)$ is starred at $x$ and $c(xz)\ne c(xy)$. By Claim 1, $y$ and $z$ are in different parts. Since $d^{c}(y)\ge 2$, there are two vertices $w_{1}, w_{2} (\ne x)$ such that $c(yw_{1})$ and $c(yw_{2})$ are starred at $y$ and $c(yw_{1})\ne c(yw_{2})$. By Claim 1, $w_{1}$ and $w_{2}$ are in different parts. Thus, there is at least one of $w_{1}$ and $w_{2}$, we say $w_{1}$, such that $z$ and $w_{1}$ are in different parts. Note that  $c(zw_{1})\notin \{c(xy), c(xz), c(yw_{1})\}$, thus $xyw_{1}zx$ is a rainbow $C_{4}$, a contradiction. \q

{\bf Claim 3.} For any vertex $x\in V$, $d^{c}(x)=2$.

{\bf Proof of Claim 3.} Assume that there is a vertex $x$ such that $d^{c}(x)\ge 3,$ there are three edges $xy, xz$ and $xw$ such that $c(xy), c(xz)$ and $c(xw)$ are starred at $x$ and distinct. By Claim 1, $y, z$ and $w$ are in different parts. Since $d^{c}(y)\ge 2$, there is a vertex $u$ such that $c(yu)$ is starred at $y$ and $x, u$ are in different parts. Also, there is  at least one of $z$ and $w$, we say $z$, such that $u, z$ are in different parts.  Note that  $c(zu)\notin \{c(xy), c(xz), c(yu)\}$, thus $xyuzx$ is a rainbow $C_{4}$, a contradiction. \q

Consider the spanning subgraph $G$ of $K_{n_{1}, n_{2}, \ldots, n_{r}}$ such that $e\in E(G)$ if and only if $e$ is the unique edge with color $c(e)$.  By Claims 1, 2 and 3, $G$ is $2$-regular and each component of $G$ is a cycle .

{\bf Claim 4.} For each component $C_{l}=x_{1}x_{2}\cdots x_{l}x_{1}$ of $G$, $x_{i}$ and $x_{i+3}$ belong to the same part of $K_{n_{1}, n_{2}, \ldots, n_{r}}$ (if $i+3>l$, $x_{i+3}$ means $x_{i+3-l}$), for all $1\le i\le l$. Furthermore, $t=\frac{n}{3}$.

{\bf Proof of Claim 4.}  Suppose that there exists $1\le i\le l$ such that $x_{i}$ and $x_{i+3}$ belong to the distinct parts of $K_{n_{1}, n_{2}, \ldots, n_{r}}$,  then $x_{i}x_{i+1}x_{i+2}x_{i+3}x_{i}$ is a rainbow $C_{4}$, a contradiction. Thus $l\equiv 0\pmod 3$ and $K_{n_{1}, n_{2}, \ldots, n_{r}}[V(C_{l})]$ contains exactly $\frac{l}{3}$ independent triangles. Since $G$ is a spanning subgraph of $K_{n_{1}, n_{2}, \ldots, n_{r}}$, we have $t=\frac{n}{3}$. \q

Since $t=\frac{n}{3}$, the maximum number of independent triangles of $K_{n_{1}, n_{2}, \ldots, n_{r}}-x$ is $\frac{n-3}{3}$, for all $x\in V$. Thus, by the induction hypothesis, we have $$|C(K_{n_{1}, n_{2}, \ldots, n_{r}}-x)|= n+t-d^{c}(x)=n+t-2= ar(K_{n_{1}, n_{2}, \ldots, n_{r}}-x, C_{4})+1,$$ and there is a rainbow $C_{4}$ in $K_{n_{1}, n_{2}, \ldots, n_{r}}-x$,  a contradiction.

\end{proof}

\n{\large\bf 4. Open problems}

We notice that the lower bound of Theorem 6 is also the lower bound for odd cycles, and we would conjecture the lower bound is the exact value when the number of the vertex of each part is sufficient large.

\begin{conjecture}
For $r\ge 3$, $k\ge 1$ and $n_{1}\ge n_{2}\ge \cdots\ge n_{r}\gg k$, we have
 \begin{align}\nonumber
  ar(K_{n_{1}, n_{2}, \ldots, n_{r}}, C_{2k+1})=
 \begin{cases}
 n_{1}n_{2}+n_{3}n_{4}+ \cdots+n_{r-2}n_{r-1}+n_{r}+\frac{r-1}{2}-1,& r\, \text{ is odd};\\
n_{1}n_{2}+n_{3}n_{4}+ \cdots+n_{r-1}n_{r}+\frac{r}{2}-1,& r \,\text{ is even}.
 \end{cases}
 \end{align}
\end{conjecture}

Also, it is interesting to investigate the anti-Ramsey number of even cycles in the complete $r$-partite graphs.

 \begin{problem}

 For $r\ge 3$ and $k\ge 3$, determine $ar(K_{n_{1}, n_{2}, \ldots, n_{r}}, C_{2k})$.

 \end{problem}

Finally we give the following conjecture which strengthens Lemma 4.

\begin{conjecture}
For $r\ge 3$ and $n_{1}\ge n_{2}\ge \cdots\ge n_{r}\ge 1$, if $G\subset K_{n_{1}, n_{2}, \ldots, n_{r}}$ and
\begin{align}\nonumber
|E(G)|\ge
 \begin{cases}
 n_{1}n_{2}+n_{3}n_{4}+ \cdots+n_{r-2}n_{r-1}+n_{r}+\frac{r-1}{2},& r\, \text{ is odd};\\
n_{1}n_{2}+n_{3}n_{4}+ \cdots+n_{r-1}n_{r}+\frac{r}{2},& r \,\text{ is even},
 \end{cases}
 \end{align}
 then $G$ contains a multipartite cycle of length no more than $\frac{3}{2}r$.
\end{conjecture}

Fang, Gy\H ori, Xiao and Xiao recently proved that  Conjecture 2 holds for $r=3$.

\end{document}